\documentclass[12pt,reqno,oneside]{amsart}
      
\usepackage{amsmath,amsthm,amsfonts,amssymb}
\usepackage{indentfirst}
\usepackage{url}
\usepackage{graphicx}
\usepackage[usenames]{color}

\theoremstyle{plain}
\newtheorem{theorem}{Theorem}[section]

\newtheorem{lem}[theorem]{Lemma}
\newtheorem{prop}[theorem]{Proposition}

\theoremstyle{definition}

\newtheorem{exa}[theorem]{Example}
\newtheorem{obs}[theorem]{Remark}

\numberwithin{equation}{section}

\newcommand{\bbN}{\mathbb{N}}

\newcommand{\mc}{C(\lambda, p)}

\newcommand{\MCO}{C_d^o(\lambda,p)}
\newcommand{\MCOd}{C_2^o(\lambda,p)}
\newcommand{\MCOt}{C_3^o(\lambda,p)}

\newcommand{\MCI}{C_d^i(\lambda,p)}
\newcommand{\MCId}{C_2^i(\lambda,p)}
\newcommand{\MCIt}{C_3^i(\lambda,p)}

\newcommand{\MCU}{C_d^u(\lambda,p)}
\newcommand{\MCUd}{C_2^u(\lambda,p)}
\newcommand{\MCUt}{C_3^u(\lambda,p)}

\newcommand{\LO}{\lambda_{c}^o(d,p)}
\newcommand{\LOd}{\lambda_{c}^o(2,p)}
\newcommand{\LOt}{\lambda_{c}^o(3,p)}

\newcommand{\LI}{\lambda_{c}^i(d,p)}
\newcommand{\LId}{\lambda_{c}^i(2,p)}
\newcommand{\LIt}{\lambda_{c}^i(3,p)}

\newcommand{\LU}{\lambda_{c}^u(d,p)}
\newcommand{\LUd}{\lambda_{c}^u(2,p)}
\newcommand{\LUt}{\lambda_{c}^u(3,p)}

\begin{document}

\baselineskip=18pt

\title[Evaluating dispersion strategies in growth models]{Evaluating dispersion strategies in growth models subject to geometric catastrophes}

\author[Valdivino Vargas Junior]{Valdivino Vargas Junior}
\address[Valdivino Vargas Junior]{Institute of Mathematics and Statistics, Federal University of Goias, Campus Samambaia, 
CEP 74001-970, Goi\^ania, GO, Brazil}
\email{vvjunior@ufg.br}
\thanks{F\'abio Machado was supported by CNPq (303699/2018-3) and Fapesp (17/10555-0) and Alejandro Roldan by Universidad de Antioquia.}

\author[F\'abio P. Machado]{F\'abio~Prates~Machado}
\address[F\'abio P. Machado]{Statistics Department, Institute of Mathematics and Statistics, University of S\~ao Paulo, CEP 05508-090, S\~ao Paulo, SP, Brazil.}
\email{fmachado@ime.usp.br}

\author[Alejandro Rold\'an]{Alejandro~Rold\'an-Correa}
\address[Alejandro Rold\'an]{Instituto de Matem\'aticas, Universidad de Antioquia, Calle 67, no 53-108, Medellin, Colombia}
\email{alejandro.roldan@udea.edu.co}


\date{\today}

\begin{abstract}
We consider stochastic growth models to represent population dynamics subject to geometric catastrophes. We analyze different dispersion schemes after catastrophes, to study how these schemes impact the population viability and comparing them with the scheme where there is no dispersion. In the schemes with dispersion, we consider that each colony, after the catastrophe event, has $d$ new positions to place its survivors.
We find out that when $d = 2$ no type of dispersion considered improves the chance of survival, at best it matches the scheme where there is no dispersion. When $d = 3$, based on the survival probability, we conclude that dispersion may be an advantage or not, depending on its type, the rate of colony growth and the probability that an individual will survive when exposed to a catastrophe.
\end{abstract}

\maketitle

\section{Introduction}
\label{S: Introduction}

Catastrophes and spatial restrictions are among biological and environmental forces that drive the size dynamics of a population. These forces can reduce the po\-pu\-lation size or even eliminate it. Dispersion of the survivors is a possible strategy that could help to increase the population viability.

Models for population growth (a single colony) subject to catastrophes are considered in Brockwell \textit{et al.}~\cite{BGR1982} and later in Artalejo \textit{et al.}~\cite{AEL2007}. In these models, the size of a colony increases according to a birth and death process subject to catastrophes. When a catastrophe strikes, the colony size is reduced according to some probability law and the survivors remain together in the same colony without dispersion. These authors study the probability distribution of the first extinction time, the number of individuals removed, the survival time of a tagged individual, and the maximum population size reached between two consecutive extinctions. For a comprehensive literature overview and motivation see for example Kapodistria \textit{et al.}~\cite{KPR2016}.

Schinazi \cite{S2014} and Machado \textit{et al}~\cite{MRS2015} study stochastic models for po\-pu\-lation growth where individuals gather in independent colonies subject to catastrophes. When a catastrophe strikes a colony the survivors disperse, trying to build new colonies that may succeed settling down depending on the environment they encounter.
Schinazi \cite{S2014} and Machado \textit{et al}~\cite{MRS2015} conclude that for these models dispersion is the best strategy. Latter Junior \textit{et al}~\cite{JMR2016} focused on models that combine two types of catastrophes (binomial and geometric) and reached the conclusion that dispersion may not be the best strategy. They observe that the best strategy depends on the type of catastrophe, the spatial restrictions that the colony must deal with and the individual survival probability when it is exposed to a catastrophe. Machado \textit{et al}~\cite{MRV2018} considered a general set up for growth rates and types of catastrophes but with severe spatial restrictions in the sense that every colony, after the catastrophe event, has only up to $d$ spots to move to.

The common point between these papers \cite{JMR2016,MRV2018,MRS2015} is the way the survivors disperse when their colony is stricken by a catastrophe. They consider that each survivor choose independently and with equal probability among the options they have to disperse. In this paper we study a variety of dispersion possibilities and show that the survival probability of the whole population may be influenced by the scheme the individuals use when the dispersion occurs.

In Section 2 we define and characterize four models for the growth of populations subject to catastrophes considering different types of dispersion. In Section 3 we compare the four models introduced in Section 2. Finally, in Section 4 we prove the results presented in Sections 2 and 3.

\section{Growth models}

Artalejo \textit{et al.}~\cite{AEL2007} present a model for a population which sticks together in one colony, without dispersion.  That colony gives birth to new individuals at rate $\lambda>0$,  while catastrophes happen at rate $\mu$. If at a catastrophe time the size of the population is $i$, it is reduced to $j$ with probability 
\[ \mu_{ij} = \left\{\begin{array}{ll} q^i, & j=0\\ pq^{i-j}, & 1\leq j \leq i,\end{array}\right.\]
where $0<p<1$ and $q=1-p$. The form of $\mu_{ij}$ represents what is called \textit{Geometric catastrophe}. Disasters reach the individuals sequentially and the effects of a disaster stop as soon as the first individual survives, if there is any survivor. The probability of next individual to survive, given that everyone fails up to that point, is $p$. 

The population size (number of individuals in the colony) at time $t$ is a continuous time Markov process $\left\{X(t):t\geq 0\right\}$ that we denote by $\mc$. We assume $\mu=1$ and $X(0)=1$. 

Artalejo \textit{et al.}~\cite{AEL2007} use the word \textit{extinction} to describe the event that $X(t) = 0$, for some $t>0$, for a process where state 0 is not an absorbing state. In fact the extinction time here is the first hitting time to the state 0. Throughout this paper we say that a process survives if the extinction probability is strictly smaller than one.

\begin{theorem}[Artalejo \textit{et al.}~\cite{AEL2007}]
	\label{th:semdisp}
	Let $X(t)$ a process $\mc$, with $\lambda>0$ and   $0<p<1$. Then, extinction event occurs with probability
	$$\psi_A=\min\left\{\frac{1-p}{\lambda p}, 1\right\}.$$
\end{theorem}

The geometric catastrophe would correspond to cases where
the decline in the population is halted as soon as
any individual survives the catastrophic event. This
may be appropriate for some forms of catastrophic
epidemics or when the  catastrophe has a sequential propagation effect like in the predator-prey models - the predator kills prey until it becomes satisfied. More examples can be found in Artalejo \textit{et al.}~\cite{AEL2007}, Cairns and Pollett~\cite{CairnsPollett}, Economou and Gomez-Corral~\cite{EG2007}, Thierry Huillet~\cite{TH2020} and Kumar \textit{et al.}~\cite{KBG2020}.

Based on the previous model we next define three models with dispersion on $\mathbb{T}_d^+$.

\subsection{Growth model with dispersion on $\mathbb{T}_d^+$.} 

Let $\mathbb{T}_d^+$ be an infinite rooted tree whose vertices
have degree $d+1$, except the root that has degree $d$. Let us define a process with 
dispersion on $\mathbb{T}_d^+$, starting from a single colony placed at the root of $\mathbb{T}_d^+$, with just one individual. The number of individuals in a colony grows following a Poisson process of rate $\lambda>0$. To each colony 
we associate an exponential time of mean 1 that indicates when the geometric catastrophe strikes a colony. The individuals that survived the catastrophe are dispersed between the $d$ neighboring vertices furthest from the root to create new colonies. Among the survivors that go to the same vertex to create a new colony at it, only one succeeds, the others die. So in this case  when a catastrophe occurs in a colony, that colony is replaced by 0,1, ... or $ d $ colonies. We consider three types of dispersion:
\begin{itemize}
	\item \textbf{Optimal dispersion:} 
	Individuals are distributed, from left to right, in order to create the largest possible number of new colonies. If $r$ individuals survive to a catastrophe, then the number of colonies that are created equals $\min\{r,d\}$. Let us denote the process with optimal dispersion by $\MCO$.
	\item 
	\textbf{Independent dispersion:} Each one of the individuals that
	survived the catastrophe picks
	randomly a neighbor vertex and tries to create a new colony at it.
	When the amount of survivors is $r$, the probability of having $y \le \min\{d,r\}$ vertices colonized
	is
	\[ \frac{T(r,y)}{d^r} {{d}\choose{y}}, \]
	where $T(r, y)$  denote the number of surjective functions $f:A \to B$, with $|A| = r$ and
	$|B| = y$.  Let us denote the process with independent dispersion by $\MCI$.
	
	\item \textbf{Uniform dispersion:} For every $r$, the amount of survivors, each 
	set of numbers $r_1, r_2, \dots, r_{d} \in \bbN$ (\textit{occupancy set of numbers}), 
	solution for
	\[ r_1 + r_2 + \cdots + r_{d} = r\]
	has probability ${{d+r-1}\choose{r}}^{-1}$. So, the probability of having $y \le \min\{d,r\}$ vertices colonized
	when the amount of survivors is $r$ is
	\[ \frac{{{r-1}\choose{y-1}}}{{{d+r-1}\choose{r}}} {{d}\choose{y}}. \]
	Let us denote the process with uniform dispersion by $\MCU$.
	
\end{itemize}

The $\MCO$, $\MCI$ and $\MCU$ are continuous-time Markov processes with state space  $\mathbb{N}_0^{\mathbb{T}^d}$. For each of these processes we say that it {\it survives}  if with positive probability there are colonies for any time in that process. Otherwise, we say that the process {\it dies out }.

We denote by $\psi_d^o$ ($\psi_d^i$ and $\psi_d^u$) the extinction probability for $\MCO$ ($\MCI$ and $\MCU$, respectively) process. By coupling arguments one can see that the extinction probability, $\psi_d^o$ ($\psi_d^i$ and $\psi_d^u$), is a non-increasing function of $d$, $\lambda$ and $p$.

\begin{obs}
	As the optimal dispersion maximizes the number of new colonies whenever there are individuals that survived from the latest catastrophe, that type of dispersion is the one which maximizes the survival probability. Moreover, for $d=2$ and $d=3$, the survival probability for the model with independent dispersion is larger or equal than the survival probability for the model with uniform dispersion. The reason for that is because the cumulative distribution function of the number of new colonies created right after a catastrophe for the model with independent dispersion is smaller than the analogous for the model with uniform dispersion. In conclusion for $d=2$ and $d=3$,
	
	\begin{equation}\label{Ext-disp}
	\psi_d^o\leq \psi_d^i \leq \psi_d^u.
	\end{equation}
\end{obs}

The next results present necessary and sufficient conditions for population survival of the process $\MCOd$. 

\begin{theorem}\label{MCOd}
The process $\MCOd$ survives ($\psi_2^o < 1 $) if and only if
\[
\lambda > \frac{1-p}{p}.
\]

Moreover
\[
\psi_2^o = \min \left \{1, \frac{1-p}{\lambda p}\right \}.
\]
\end{theorem}

\begin{obs}\label{surprising} From Theorems~\ref{th:semdisp} and~\ref{MCOd} one sees that extinction probabilities for the models $\mc$ and $\MCOd$, are the same ($\psi_A=\psi_2^o$). This is a big surprise and don't see any intuitive reason for this. As a consequence, when $d=2$, for fixed $\lambda$ and $p$, it is not possible to increase the survival probability by any type of dispersion.
\end{obs}

\begin{theorem}\label{MCOt}
The process $\MCOt$ survives ($\psi_3^o < 1 $) if and only if
\[
p > \frac{\lambda + 1}{2 \lambda^2 + 2 \lambda + 1}.
\]
Moreover
\[
\psi_3^o = \min \left \{1, \frac{\lambda + 1}{2 \lambda} \displaystyle \left [ -1 + \sqrt{\frac{\lambda p +4 - 3p}{(\lambda +1)p}}  \right ] \right \}.
\]
\end{theorem}

\begin{theorem}\label{MCId}
The process $\MCId$ survives ($\psi_2^i < 1 $) if and only if
\[
p > \frac{\lambda + 2}{\lambda^2 + 2 \lambda + 2}.
\]
Moreover
\[
\psi_2^i =  \min \left \{1, \frac{(1-p)(\lambda +2)}{\lambda p(\lambda+1)} \right \}.
\]
\end{theorem}

\begin{theorem}\label{MCIt}
The process $\MCIt$ survives ($\psi_3^i < 1 $) if and only if

\begin{equation}\label{lambda3i}
p > \frac{\lambda + 3}{2\lambda^2 + 3 \lambda + 3}.
\end{equation}
Moreover
\[
\psi_3^i = \min \left \{1, \frac{1}{2 \lambda} \displaystyle \left [ -(\lambda +3) + \sqrt{\frac{(\lambda +3)(p\lambda^2 +4 \lambda +6-3p)}{p(\lambda +1)}} \right ] \right \}.
\]

\end{theorem}

\begin{theorem}\label{MCUd} 
The process $\MCUd$ survives ($\psi_2^u < 1 $) if and only if

\[
\ln(\lambda + 1) < \frac{\lambda[ \lambda^2 p + (4p-1)\lambda + 2p]}{2 (\lambda +1 )^2 p}.
\]
Moreover
\[
\psi_2^u = \min \left \{1, \frac{\lambda^2 (1-p)}{(\lambda +2)(\lambda +1)\lambda p - 2p(\lambda+1)^2 \ln(\lambda +1)} \right \}.
\]	
\end{theorem}

\begin{theorem}\label{MCUt} 
The process $\MCUt$ survives ($\psi_3^u < 1 $) if and only if

\begin{equation}\label{lambda3u}
\frac{3p(\lambda+1)^2}{\lambda^2 (\lambda p +1 )}\displaystyle \left [ \lambda + 2 - \frac{2(\lambda +1)}{\lambda} \ln (\lambda +1) \right] > 1. 
\end{equation}
Moreover
\[
\psi_3^u = \min \left \{1, \frac{1}{2} \displaystyle \left[ \frac{ \sqrt {\Delta} -(k_2+k_3) +(m_2+m_3) \ln(\lambda+1)}{k_3 -m_3 \ln (\lambda+1)}  \right ] \right \},
\]
where
\[ \Delta = (m_2 +m_3)^2 \ln^2 (\lambda + 1) - 2 [(k_2+k_3)(m_2+m_3) + 2\beta m_3] \ln(\lambda+1) + [(k_2+k_3)^2 +4\beta k_3],
\]

\[ \beta= \frac{1-p}{\lambda p +1},\, k_2 = \frac{-3p(\lambda +1)(5 \lambda +6)}{\lambda^2 (\lambda p +1)},\, k_3 = \frac{p(\lambda +1)(\lambda^2 +12 \lambda +12)}{\lambda^2 (\lambda p +1)}, 
\]
\[m_2 = \frac{-6p (\lambda + 1)^2 (\lambda + 3)}{\lambda^3 (\lambda p+1)} \text{ and }  m_3 = \frac{6p(\lambda +1)^2(\lambda + 2)}{\lambda^3 (\lambda p + 1)}.
\]	
\end{theorem}

\section{Dispersion as a survival strategy}

Towards being able to evaluate dispersion as a survival strategy we define

$$\lambda_c(p):=\inf\{\lambda:  \mc \text{ survives} \},$$
$$\LO:=\inf\{\lambda: \MCO \text{ survives} \},$$
$$\LI:=\inf\{\lambda: \MCI \text{ survives} \},$$
$$\LU:=\inf\{\lambda: \MCU \text{ survives} \}.$$

\begin{obs} When $0<\lambda_c(p)<\infty$ for  $0<p<1,$ the graph 
of $\lambda_c(p)$ splits the parametric space $\lambda \times p$ into two regions. For those values of $(\lambda,p)$ above the curve $\lambda_c(p)$  there is survival in $\mc$ with positive probability, and for those values of $(\lambda,p)$  below the curve  $\lambda_c(p)$ extinction occurs in $\mc$ with probability 1. The analogous happens also for $\LO$, $\LI$ and  $\LU$. For $d = 2$ and $d = 3$, from inequality~(\ref{Ext-disp}) it follows that $$\LO\leq\LI\leq\LU.$$ For an illustration, see Figures \ref{fig:sub1} and \ref{fig:sub2}.
\end{obs}

\begin{figure}[ht]
	\begin{tabular}{ccc}
		$\lambda$ & \parbox[c]{9cm}{\includegraphics[trim={1.1cm 1.5cm 1cm 1cm}, clip, width=9cm]{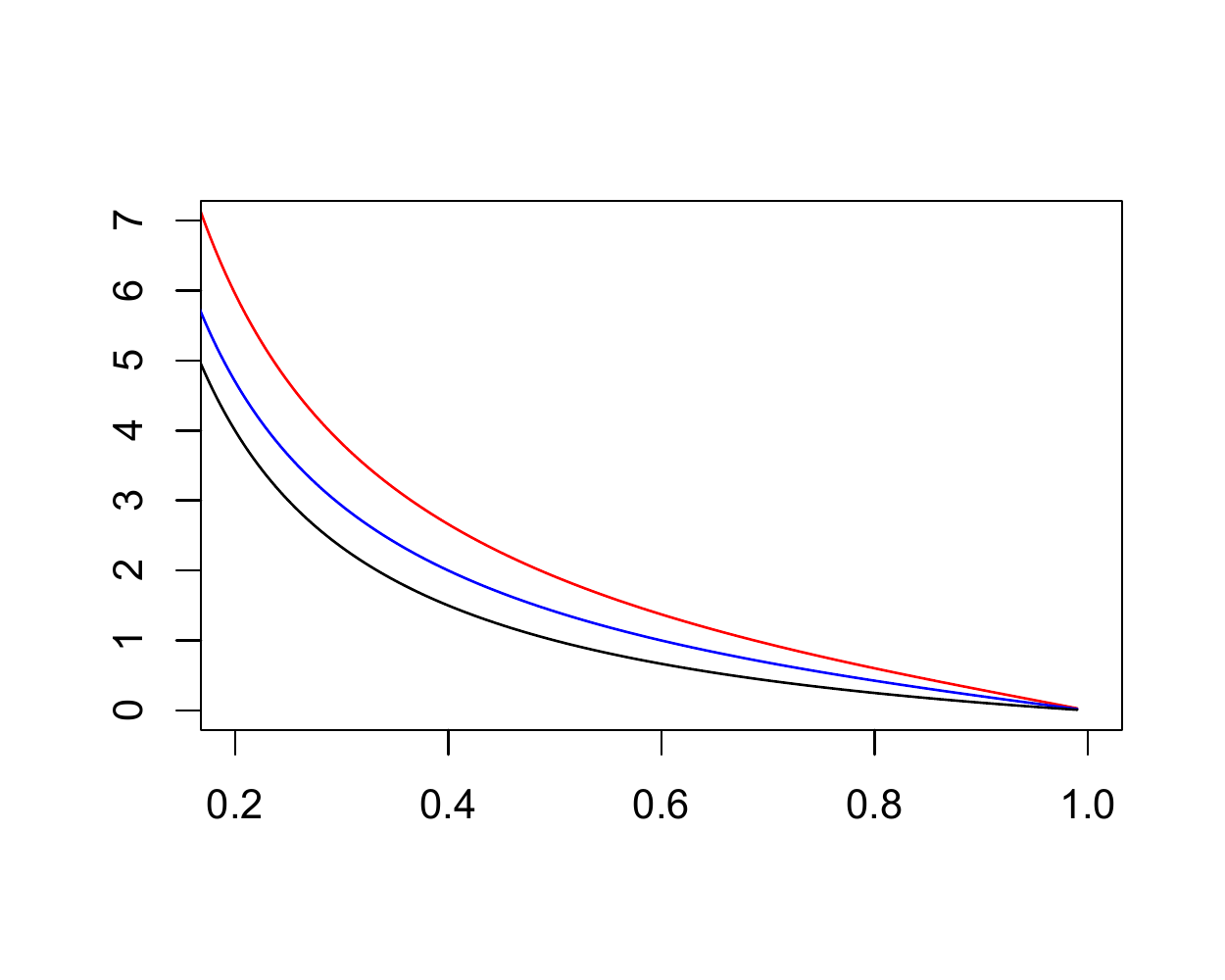}} & 
		\begin{tabular}{l}
			\textbf{\textcolor{red}{------}} $\LUd$\\ 
			\textbf{\textcolor{blue}{------}} $\LId$\\
			\textbf{\textcolor{black}{------}} $\LOd$\\
			\textbf{\textcolor{black}{------}} $\lambda_c(p)$
		\end{tabular} \\
		& $p$
	\end{tabular}
	\caption{Graphics of $\lambda_c(p)$, $\LOd$, $\LId$ and $\LUd$. }
	\label{fig:sub1}
\end{figure}

\begin{figure}[ht]
	\begin{tabular}{ccc}
		$\lambda$ & \parbox[c]{9cm}{\includegraphics[trim={1.1cm 1.5cm 1cm 1cm}, clip, width=9cm]{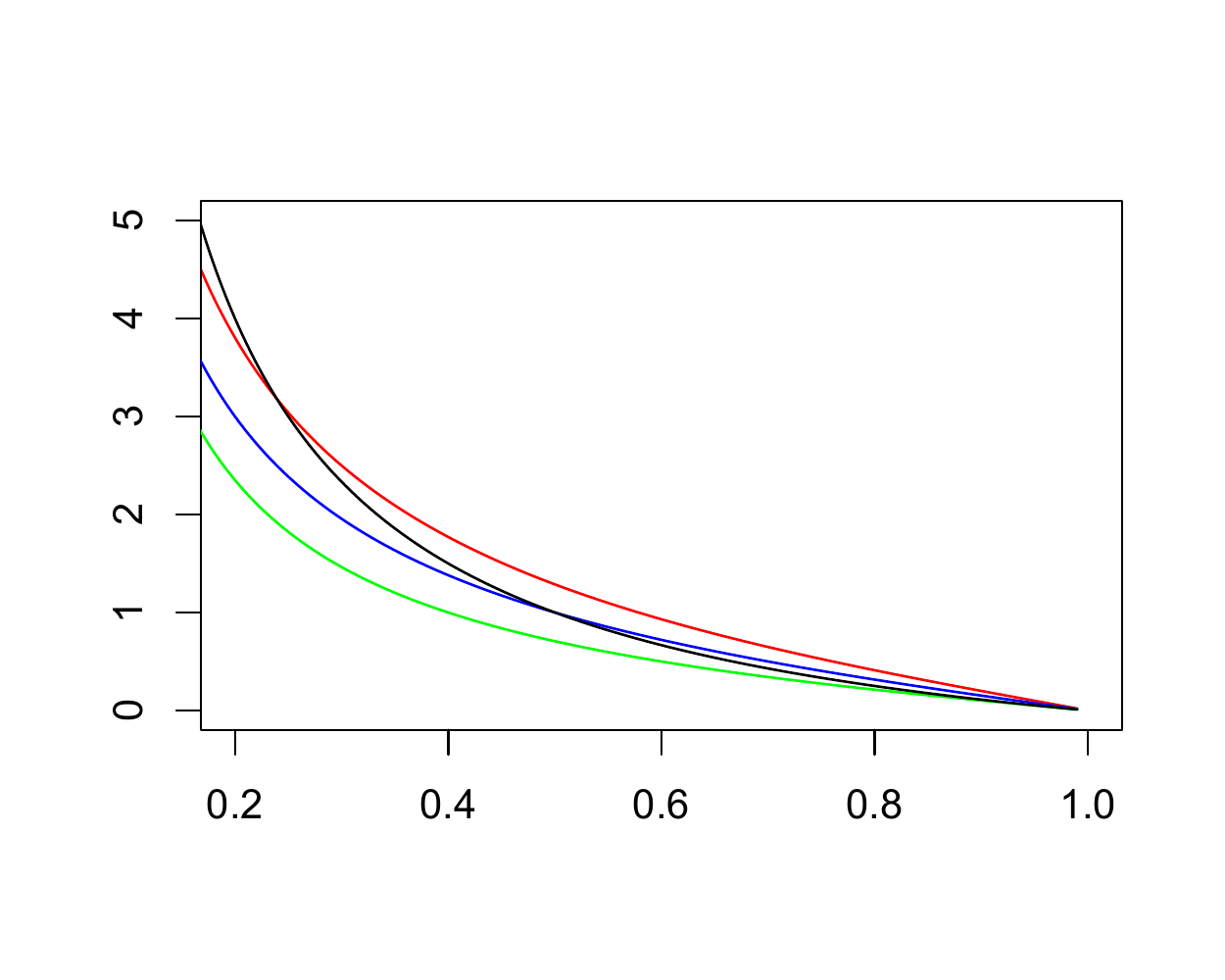}} & 
		\begin{tabular}{l}
			\textbf{\textcolor{red}{------}} $\LUt$\\ 
			\textbf{\textcolor{blue}{------}} $\LIt$\\
			\textbf{\textcolor{green}{------}} $\LOt$\\
			\textbf{\textcolor{black}{------}} $\lambda_c(p)$
		\end{tabular} \\
		& $p$
	\end{tabular}
	\caption{Graphics of $\lambda_c(p)$, $\LOt$, $\LIt$ and $\LUt$. }
	\label{fig:sub2}
\end{figure}

From Remark~\ref{surprising} and the fact that the extinction probabilty, $\psi_d^o$ ($\psi_d^i$ and $\psi_d^u$), is a non-increasing function of $d$, $\lambda$ and $p$, one sees that $\psi_3^o \leq \psi_2^o = \psi_A$. When analysing the critical parameters  one sees (Figure~\ref{fig:sub2}) that $\LOt < \lambda_c(p)$. This shows that when $d=3$, for all $0<p<1$, the optimal dispersion is a superior strategy when compared to the non-dispersion scheme studied in Artalejo \textit{et al.}~\cite{AEL2007}. 
	 
However, dispersion is not always a better scenary for population survival, as one can see in Figure \ref{fig:sub2}. 
Observe that:
\begin{equation}\label{pi3c}
\LIt\leq\lambda_c(p) \iff p\leq p_3^i=\frac{1}{2}
\end{equation}
and
\begin{equation}\label{pu3c}
\LUt\leq\lambda_c(p) \iff p\leq  p_3^u \approx 0.239139,
\end{equation}

\noindent
where the values $p^i_3$ and $p^u_3$ are obtained by plugging $\lambda = \lambda_c(p)=\frac{1-p}{p}$ in equations (\ref{lambda3i}) and (\ref{lambda3u}), respectively, taken as equality.

In the region bounded by the curves $\LIt$ and $\lambda_c(p)$, for $0< p< p_3^i$, we obtain from (\ref{pi3c}) that $\psi_3^i<\psi_A=1$. In this case,  \textit{independent dispersion} is a better strategy than \textit{non-dispersion}. On the other hand, in the region bounded by the curves $\LIt$ and $\lambda_c(p)$ when $p_3^i<p<1$, we obtain that $\psi_A<\psi_3^i=1$ and that \textit{non-dispersion} is a better than \textit{independent dispersion}. These two latter observations were also presented in Junior~\textit{et al}~\cite{JMR2016}. 
 	
Analogously, we can conclude from (\ref{pu3c}) that \textit{uniform dispersion} is a better strategy than \textit{non-dispersion} in the region bounded by $\LUt$ and $\lambda_c(p)$ when $0< p< p_3^u$. The opposite (\textit{non-dispersion} is a better strategy than \textit{uniform dispersion}) holds in the region bounded by $\LUt$ and $\lambda_c(p)$ when $p_3^u<p<1$.
 
An interesting question is to determine whether dispersion is an advantage or not for population survival in the region of parameters where both processes $\mc$ and $\MCIt$ (or $\MCUt$) survives. This question is answered by the following propositions.

\begin{prop}\label{indep3}
	Assume that $\psi_3^i<1$ and  $\psi_A<1$.  Then,
$\psi_3^i<\psi_A$ if and only if   
\begin{equation}\label{eq:indep}
p<\frac{2(\lambda+1)}{3\lambda+5}.
\end{equation}
\end{prop}

Proposition~\ref{indep3} is a direct consequence of Theorems \ref{th:semdisp} and \ref{MCIt}.From~{(\ref{pi3c})} and Proposition~\ref{indep3} we can conclude that \textit{independent dispersion} is a better strategy compared to \textit{non-dispersion}, when the parameters $(\lambda, p)$ fall in the gray region of Figure~\ref{fig:sub3}. The opposite (\textit{non-dispersion} is a better strategy than \textit{independent dispersion}) holds in the yellow region. 
\begin{figure}[h!]
	\begin{tabular}{ccc}
		$\lambda$ & \parbox[c]{9cm}{\includegraphics[trim={0cm 0cm 0cm 0cm}, clip, width=9cm]{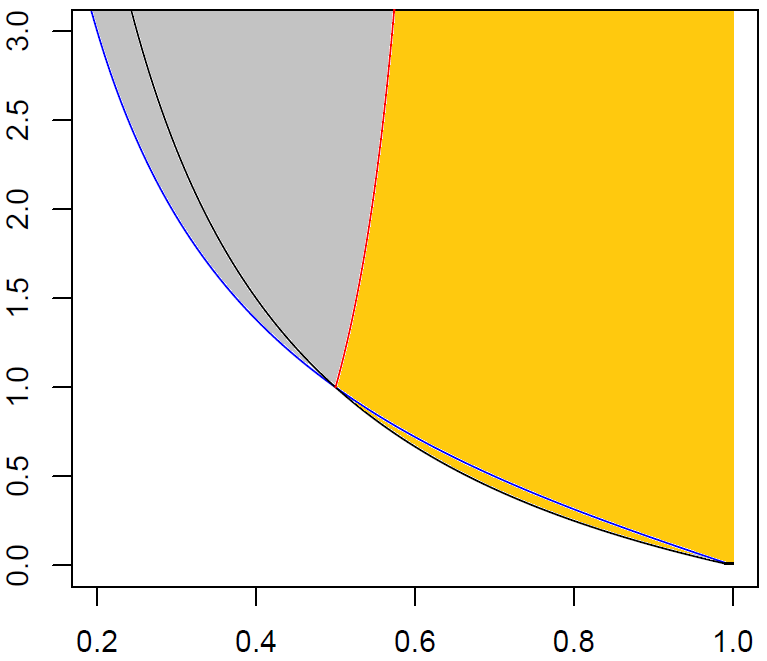}} & 
		\begin{tabular}{l} 
			\textbf{\textcolor{blue}{------}} $\LIt$\\
			\textbf{\textcolor{black}{------}} $\lambda_c(p)$\\
			\textbf{\textcolor{red}{------}} Equality in (\ref{eq:indep})\\ 
		\end{tabular} \\
		& $p$
	\end{tabular}
	\caption{\textit{Independent dispersion} vs \textit{non-dispersion}. In the gray region, $\psi_3^i<\psi_A$. In the yellow region, $\psi_3^i>\psi_A$.}
	\label{fig:sub3}
\end{figure}

\begin{obs}\label{cases} Consider the processes $C(\lambda,p)$ and $C_3^i(\lambda,p)$. From Theorems~\ref{th:semdisp} and \ref{MCIt} and Proposition~\ref{indep3} it follows that, 
	
\begin{itemize}
	\item For $\lambda>1$

\begin{itemize}
	\item If $p\in(0,\frac{\lambda + 3}{2\lambda^2 + 3 \lambda + 3}]$, then $\psi_A=\psi_3^i=1$.
	\item If $p\in(\frac{\lambda + 3}{2\lambda^2 + 3 \lambda + 3},\frac{1}{\lambda+1}]$, then $\psi_3^i<\psi_A=1$.
	\item If $p\in(\frac{1}{\lambda+1},\frac{2(\lambda+1)}{3\lambda+5})$, then $\psi_3^i<\psi_A<1$.
	\item If $p=\frac{2(\lambda+1)}{3\lambda+5}$, then  $\psi_3^i=\psi_A=\frac{\lambda +3}{2\lambda(\lambda+1)}<1$.
	\item If $p\in(\frac{2(\lambda+1)}{3\lambda+5},1)$, then $\psi_A<\psi_3^i<1$.
\end{itemize}

\item For $\lambda=1$
\begin{itemize}
    \item If $p\in(0,\frac{1}{2}]$, then $\psi_A=\psi_3^i=1$.
    \item If $p\in(\frac{1}{2},1]$, then $\psi_A<\psi_3^i<1$.
\end{itemize}

\item For $\lambda<1$
\begin{itemize}
	\item If $p\in(0,\frac{1}{\lambda + 1}]$, then $\psi_A=\psi_3^i=1$.
	\item If $p\in(\frac{1}{\lambda + 1}, \frac{\lambda + 3}{2\lambda^2 + 3 \lambda + 3}]$, then $\psi_A< \psi_3^i=1$.
	\item If $p\in(\frac{\lambda + 3}{2\lambda^2 + 3 \lambda + 3},1)$, then $\psi_3^i<\psi_A<1$.
\end{itemize}
\end{itemize}
\end{obs}

\begin{prop}\label{unif3}
	Assume that $\psi_3^u<1$ and  $\psi_A<1$.  Then, $\psi_3^u<\psi_A$ if and only if
	\begin{equation}\label{eq:unif}
	   \left[\left(\frac{1+p(\lambda-1)}{\lambda p}\right)m_3+m_2\right]\ln(\lambda+1)<
	\left(\frac{1+p(\lambda-1)}{\lambda p}\right)k_3+k_2-\frac{\lambda p}{\lambda p +1}.
	\end{equation}
\end{prop}

Proposition~\ref{unif3} is a consequence of Theorems \ref{th:semdisp} and \ref{MCUt}. From~{(\ref{pu3c})} and Proposition~\ref{unif3} we can conclude that \textit{uniform dispersion} is a better strategy compared to \textit{non-dispersion}, when the parameters $(\lambda, p)$ fall in the gray region of Figure~\ref{fig:sub4}. The opposite (\textit{non-dispersion} is a better strategy than \textit{uniform dispersion}) holds in the yellow region. Analogous to what is presented in Remark~\ref{cases} we have $\lambda \approx 3.18$ splitting the cases.

\begin{figure}[h!]
	\begin{tabular}{ccc}
		$\lambda$ & \parbox[c]{9cm}{\includegraphics[trim={0cm 0cm 0cm 0cm}, clip, width=9cm]{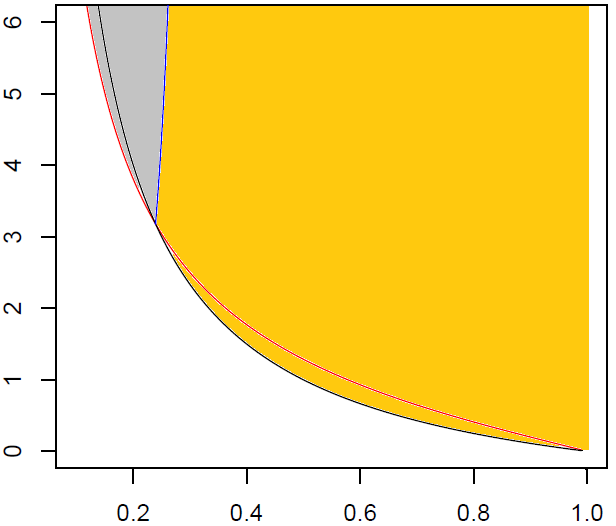}} & 
		\begin{tabular}{l} 
			\textbf{\textcolor{red}{------}} $\LUt$\\
			\textbf{\textcolor{black}{------}} $\lambda_c(p)$\\
			\textbf{\textcolor{blue}{------}} Equality in (\ref{eq:unif})
		\end{tabular} \\
		& $p$
	\end{tabular}
	\caption{\textit{Uniform dispersion} vs \textit{non-dispersion}. In the gray region, $\psi_3^u<\psi_A$. In the yellow region, $\psi_3^u>\psi_A$.}
	\label{fig:sub4}
\end{figure}

\begin{exa} Consider the processes $C(4,p)$ and $C_3^u(4,p)$. From Theorems~\ref{th:semdisp} and \ref{MCUt} and Proposition~\ref{unif3} follows that
	\begin{itemize}
		\item If $p\in\left(0 , \frac{32}{772-375\ln(5)}\right]$, then $\psi_A=\psi_3^u=1$.
		\item If $p\in\left(\frac{32}{772-375\ln(5)},\frac{1}{5}\right]$, then $\psi_3^u<\psi_A=1$.
		\item If $p\in\left(\frac{1}{5},\frac{380-225\ln(5)}{676-375\ln(5)}\right)$, then $\psi_3^u<\psi_A<1$.
		\item If $p=\frac{380-225\ln(5)}{676-375\ln(5)}$, then  $\psi_3^i=\psi_A<1$.
		\item If $p\in\left(\frac{380-225\ln(5)}{676-375\ln(5)},1\right)$, then $\psi_A<\psi_3^u<1$.
	\end{itemize}	
\end{exa}

\begin{obs} Observe that the model with only one colony has a catastrophe rate of $1$ while the multiple colonies model has a catastrophe rate of $n$ if there are $n$ colonies. Moreover, a catastrophe is more likely to wipe out a smaller colony than a larger one. On the other hand multiple colonies give multiple chances for survival and this may be a critical advantage of the multiple colonies model over the single colony model. Therefore our analysis shows that for $d=3$, dispersion is an advantage or not for population survival depending on the dispersion type, $\lambda$ and $p$.
\end{obs}

\section{Proofs}

\begin{lem}\label{lemaux}
	Let $\{Z_n\}_{n \geq 0}$ be a branching process with $Z_0 = 1$, whose offspring distribution has probabilty generating function given by $g(s) = p_0 +p_1s+ p_2s^2 + p_3 s^3$  and with extinction probability denoted by $\psi$. Then\\
	\begin{itemize}
		\item[$i)$] $\psi < 1$ if and only if $p_1+ 2p_2 + 3p_3 > 1.$
		\item[$ii)$] Assume $p_3 \neq 0$. If $p_1+ 2p_2 + 3p_3 > 1$ then
		\[ \psi = \frac{1}{2} \displaystyle \left [ -1 - \frac{p_2}{p_3} + \sqrt{ \left (1 + \frac{p_2}{p_3}\right )^2 + \frac{4p_0}{p_3}}  \right].
		\]
		\item[$iii)$] Assume $p_3=0$. If $p_1 +2p_2 > 1$ then
		\[ \psi = \frac{p_0}{p_2}.
		\]
		\item[$iv)$] Assume $p_0 = \beta$ and $p_y = k_y + m_y \ln \nu, y=1,2,3.$ Then 
		\[ \psi < 1 \hbox{ if and only if } (m_1 +2m_2 +3m_3)\ln \nu < 1-k_1-2k_2-3k_3. 
		\]
		If $(m_1 +2m_2 +3m_3)\ln \nu < 1-k_1-2k_2-3k_3$ and $p_3 \neq 0$ then
		\[ \psi = \frac{1}{2} \displaystyle \left [ -1 + \frac{\sqrt{\Delta} -k_2 -m_2 -\ln \nu}{k_3 + m_3 \ln \nu} \right ] 
		\] 
		where
		\[ \Delta = (m_2+m_3)^2 \ln^2 \nu + 2[( k_2 +k_3)(m_2+m_3) + 2 \beta m_3 ] \ln \nu + [(k_2 +k_3)^2 +4 \beta k_3].
		\]
		\item[$v)$] Take $k_3=m_3 = 0$ in $iv)$. Then
		\[ \psi < 1 \hbox{ if and only if } (m_1+2m_2) \ln \nu < 1-k_1 - 2k_2.
		\]
		Besides, if $(m_1 +2m_2)\ln \nu < 1-k_1-2k_2$ and $p_2 \neq 0$ then
		\[ \psi = \frac{\beta}{k_2 +m_2 \ln \nu}.
		\]
	\end{itemize}
\end{lem}

\begin{proof}[Proof of Lemma \ref{lemaux}]\text{}\\
	
\begin{itemize}
\item[$i)$] $\psi < 1$ if and only if $ \mu = g'(1) > 1$ where $g'(1) = p_1 + 2 p_2 + 3 p_3.$
\item[$ii)$] $\psi$ is the smallest non negative solution of $g(s) = s$. That is, one must solve the equation
\[	(s-1)(p_3s^2 +(p_2+p_3)s -p_0) = 0.\]
\item[$iii)$] $\psi$ is the smallest non negative solution of $g(s) = s$. Then when $p_3 = 0$, one has to solve the equation 
\[ (s-1)(p_2s -p_0) = 0.
\] 
\item[$iv)$] The first part follows immediately from $i)$. For the second part, observe that in $ii)$
\[ \psi = \frac{1}{2} \displaystyle \left [ -1 + \frac{\sqrt{(p_2+p_3)^2 + 4p_0p_3} - p_2}{p_3} \right].
\]
\item[$v)$] The first part follows from $i)$ and the second part follows from $iii)$.
\end{itemize}
\end{proof}

\begin{obs}
The probability distribution of the number of survivals right after the catastrophe (but before the dispersion) is given by
\[\mathbb{P}(N=0) =\beta, \,\mathbb{P}(N = n) = \alpha c^n, n =1,2,\ldots,\]  where
\[ \beta = \frac{1-p}{\lambda p + 1}, \ \alpha = \frac{(\lambda + 1)p}{\lambda ( \lambda p +1)} \hbox{ and } \ c = \frac{\lambda}{\lambda + 1} .
\]
Observe that
\begin{equation}\label{eqBeta}
\beta = \frac{1-(1+ \alpha) c}{(1-c)}.
\end{equation}

For details see Machado~\textit{et al}~\cite[section 2.2]{MRV2018}.
\end{obs}

\begin{proof}[Proof of Theorem \ref{MCOd}]

Observe that for $d=2$ we have
\[
p_0 = \mathbb{P}(N=0) = \beta, \ p_1 = \mathbb{P}(N=1)= \alpha c \hbox{ and } p_2 = 1- \beta - \alpha c.
\]
Next, apply Lemma~\ref{lemaux}, part $iii)$.
\end{proof}

\begin{proof}[Proof of Theorem \ref{MCOt}]

Here $d=3$, then we have
\[
p_0 = \mathbb{P}(N=0) = \beta, \ p_1 = \mathbb{P}(N=1)= \alpha c, \ p_2 =  \mathbb{P}(N=2)= \alpha c^2 \hbox{ and } p_3 = 1- \beta - \alpha c - \alpha c^2.
\]
Next, applying Lemma~\ref{lemaux}, part $ii)$

$\begin{array}{rl}
\psi =& \displaystyle\frac{1}{2} \displaystyle \left [ - \left ( 1 + \frac{\alpha c^2}{1- \beta - \alpha c - \alpha c^2}  \right ) + \sqrt {\left ( 1 + \frac{\alpha c^2}{1- \beta - \alpha c - \alpha c^2}  \right )^2 + \frac{4\beta}{1- \beta - \alpha c - \alpha c^2}} \right ]\\

 =& \displaystyle\frac{1}{2(1- \beta - \alpha c - \alpha c^2)} \displaystyle \left [ - (1-\beta-\alpha c) + \sqrt{(1- \beta - \alpha c)^2 + 4 \beta(1- \beta - \alpha c - \alpha c^2)} \right ].
\end{array}$

\noindent
The result follows when one considers
\[ \beta = \frac{1-p}{\lambda p + 1}, \ \alpha = \frac{(\lambda + 1)p}{\lambda ( \lambda p +1)} \hbox{ and } c = \frac{\lambda}{\lambda + 1}.
\]
\end{proof}

\begin{proof}[Proof of Theorem \ref{MCId}]
The first part can be seen as an application of Proposition 4.3 from Machado~\textit{et al}~\cite{MRV2018}. As $d=2$,
	\[ \psi^i_2 < 1 \hbox{ if and only if } \mathbb{E} \left [ \left ( \frac{1}{2}  \right )^N \right ] < \frac{1}{2}.
	\]
Observe that 
\begin{equation}\label{fgpN}
\mathbb{E} \left(s^N \right) = \beta + \sum_{n=1}^{\infty}s^n \alpha c^n = \frac{sc(\alpha-\beta) + \beta}{1-sc}.
\end{equation}
	Then, considering Proposition 4.3 from Machado~\textit{et al}~\cite{MRV2018} we see that $\psi^i_2 < 1$ if and only if
	\[ \frac{\frac{1}{2}c(\alpha-\beta) + \beta}{1-\frac{1}{2}c} < \frac{1}{2},
	\]

	then, substituting $\alpha, \beta$ and $c$ accordingly we see that $\psi^i_2 < 1$ if and only if
	\[
	p > \frac{\lambda + 2}{\lambda^2 + 2 \lambda + 2}.
	\]
In order to obtain the extinction probability we consider Proposition 4.4 from Machado~\textit{et al}~\cite{MRV2018}. There we have that when $d=2$, $\psi^i_2$ is the smallest non-negative solution of
	\[ \sum_{y=0}^{2} \left [ s^y \dbinom{2}{y}\sum_{n=y}^{\infty} \frac{T(n,y)}{2^n} \mathbb{P}(N=n) \right ] = s
	\]
	where
	\[ T(n,y) = \sum_{i=0}^{y} \left [ (-1)^y \dbinom{y}{i}(y-i)^n \right ].
	\]
	Then, to obtain $\psi^i_2$ we have to solve the equation
	\[ \beta + 2s \alpha \sum_{n=1}^{\infty} \left (\frac{1}{2} \right )^n \alpha c^n  + s^2 \sum_{n=2}^{\infty} \left ( \frac{2^n-2}{2^n} \right )\alpha c^n = s
	\]
	or equivalently
	\[ \beta + \left (\frac{2 \alpha c + c - 2}{2-c} \right)s + \frac{\alpha c^2}{(1-c)(2-c)}s^2 = 0.
	\]
	Then, substituting $\alpha, \beta$ and $c$ accordingly, we obtain
	\[ \lambda p (\lambda +1)s^2 +(2p-\lambda^2 p- \lambda -2)s +(1-p)(\lambda +2) = 0
	\]
	or equivalently
	\[
	(s-1)[\lambda p(\lambda+1)s + (\lambda +2)(p-1)] = 0.
	\]
	Finally, we have that
	\[
	\psi_2^i = \min\left\{ 1, \frac{(1-p)(\lambda +2)}{\lambda p(\lambda+1)}\right\}.
	\]
\end{proof}

\begin{proof}[Proof of Theorem \ref{MCIt}]
	
	Proposition 4.3 from Machado~\textit{et al}~\cite{MRV2018} ($d=3$) is to be considered in order to prove the first part. First of all observe that
		\[ \psi^i_3 < 1 \hbox{ if and only if } \mathbb{E} \left [ \left ( \frac{2}{3}  \right )^N \right ] < \frac{2}{3}.
		\]

Using equation~(\ref{fgpN}) and applying Proposition 4.3 from Machado~\textit{et al}~\cite{MRV2018} we see that $\psi^i_3 < 1$ if and only if
		\[ \frac{c}{2} \left (\alpha-\beta \right) +\beta < \frac{1}{2} \left(1- \frac{c}{2} \right).
		\]
		
Then, substituting $\alpha, \beta$ and $c$ accordingly, we see that $\psi^i_3 < 1$ if and only if
		\[
		p > \frac{\lambda + 3}{2\lambda^2 + 3 \lambda + 3}.
		\]

In order to obtain the extinction probability we consider Proposition 4.4 from Machado~\textit{et al}~\cite{MRV2018}. There we have that when $d=3$, $\psi^i_3$ is the smallest non-negative solution of

		\[ \beta + 3\alpha s \left [ \sum_{n=1}^{\infty} \left (\frac{c}{3}\right )^n  \right ]
		+ 3\alpha s^2\left[ \sum_{n=2}^{\infty}\frac{(2^n -2)}{3^n} c^n \right] + \alpha s^3\left [ \sum_{n=3}^{\infty}\frac{(3^n-3\cdot2^n+3)}{3^n}c^n \right ] = s.
		\]
		So, we have to solve
		\[ 2 \alpha c^3 s^3 + 6 \alpha c^2(1-c)s^2 + (3 \alpha c +c-3)(1-c)(3-2c)s +  \beta (1-c)(3-2c)(3-c) = 0,
		\]
		or equivalently (see equation~(\ref{eqBeta})), solve
		\[ (s-1)[(2 \alpha c^3)s^2 + (6 \alpha c^2 -4 \alpha c^3)s + 2 \alpha c^3 -9 \alpha c^2 +9 \alpha c + 2c^3 -11 c^2 +18 c-9] = 0.
		\]
		Substituting $\alpha$ and $c$, equivalently we have to solve
		\[ (s-1)\left[s^2 + \frac{(\lambda + 3)}{\lambda} s + \frac{(p-1)(2 \lambda^2 + 9 \lambda +9)}{2(\lambda +1)p \lambda^2}\right] = 0
		\]
		obtaining
		\[
		\psi^i_3 = \min\left\{1, \frac{1}{2 \lambda} \displaystyle \left [ -(\lambda +3) + \sqrt{\frac{(\lambda +3)(p\lambda^2 +4 \lambda +6-3p)}{p(\lambda +1)}} \right]\right\}.
		\]\end{proof}

\begin{proof}[Proof of Theorem \ref{MCUd}]
	Observe that the process $\MCUd$ behaves a branching process. Next we show that the distribution of $Y$, the number of new colonies right after a catastrophe, satisfies item $v)$ of Lemma \ref{lemaux}.
	First we see that
	\[ \mathbb{P}(Y = 0 ) = \mathbb{P}(N = 0 ) =  \frac{1-p}{\lambda p +1}.
	\]
	For $n \neq 0$,
	\[ \mathbb{P}(Y = y | N=n) = \frac{\dbinom{2}{y} \dbinom{n-1}{y-1}}{{n+1}}, y = 1,2;\ y \leq n.
	\]

	Then,	
	\[  \mathbb{P}(Y = 1 ) = \sum_{n=1}^{\infty} \mathbb{P}(N=n)\mathbb{P}(Y = 1| N=n)=-2 \alpha - \frac{2 \alpha}{c} \ln (1-c)
	\]
	and
	\[  \mathbb{P}(Y = 2 ) = \sum_{n=2}^{\infty} \mathbb{P}(N=n)\mathbb{P}(Y = 2| N=n)= \frac{\alpha (2-c)}{1-c} \frac{2 \alpha }{c} \ln(1-c).
	\]
	
	Substituting $\alpha$ and $c$ we see that
	\[\mathbb{P}(Y=1) = \frac{-2p(\lambda +1)}{\lambda( \lambda p +1)}  - \frac{2p(\lambda +1)^2}{\lambda^2(\lambda p +1)}\ln \displaystyle \left (\frac{1}{\lambda +1} \right)
	\]
	and
	\[ \mathbb{P}(Y=2) = \frac{p(\lambda +1)(\lambda +2)}{\lambda (\lambda p+1)} + \frac{2(\lambda +1)^2 p}{\lambda^2 (\lambda p+1)} \ln  \left ( \frac{1}{\lambda +1} \right). 
	\] 
	The result follows after identifying the parameters $k_y, m_y$ and $\nu$ in item $v)$ of Lemma \ref{lemaux}.
	
\end{proof}

\begin{proof}[Proof of Theorem \ref{MCUt}] Observe that the process $\MCUt$ behaves a branching process. Next we show that the distribution of $Y$, the number of new colonies right after a catastrophe, satisfies item $iv)$ of Lemma \ref{lemaux}. First we see that
	\[ \mathbb{P}(Y = 0 ) = \mathbb{P}(N = 0 ) =  \frac{1-p}{\lambda p +1}.
	\]
	For $n \neq 0$:
\[ \mathbb{P}(Y = y | N=n) = \frac{\dbinom{3}{y} \dbinom{n-1}{y-1}}{\dbinom{n+2}{2}}, y = 1,2,3, y \leq n.
\]

Then,
\[  \mathbb{P}(Y = 1 ) = \sum_{n=1}^{\infty} \mathbb{P}(N=n)\mathbb{P}(Y = 1| N=n)= \frac{3 \alpha (2-c)}{c} + \frac{6 \alpha (1-c)}{c^2} \ln(1-c),
\]
\[  \mathbb{P}(Y = 2 ) = \sum_{n=2}^{\infty} \mathbb{P}(N=n)\mathbb{P}(Y = 2| N=n)= \frac{3 \alpha (c-6)}{c} + \frac{12 \alpha (2c-3)}{c^2} \ln(1-c)
\]
and
\[  \mathbb{P}(Y = 3 ) = \sum_{n=3}^{\infty} \mathbb{P}(N=n)\mathbb{P}(Y = 3| N=n)= \frac{ \alpha (c^2 -12c+12)}{c(1-c)} + \frac{6 \alpha (2-c)}{c^2} \ln(1-c).
\]

Substituting $\alpha$ and $c$ we see that
\[  \mathbb{P}(Y = 1) =  \frac{3p(\lambda +1)(\lambda +2)}{\lambda^2 (\lambda p +1)} + \frac{6p(\lambda+1)^2}{\lambda^3 (\lambda p +1)} \ln \left ( \frac{1}{\lambda + 1} \right ),
\]
\[  \mathbb{P}(Y = 2) =  \frac{-3p(\lambda +1)(5 \lambda +6)}{\lambda^2 (\lambda p +1)} - \frac{6p(\lambda+1)^2(\lambda +3)}{\lambda^3 (\lambda p +1)} \ln \left ( \frac{1}{\lambda + 1} \right )
\]
and
\[  \mathbb{P}(Y = 3) =  \frac{p(\lambda +1)(\lambda^2 + 12 \lambda +12)}{\lambda^2 (\lambda p +1)} + \frac{6p(\lambda+1)^2 (\lambda +2)}{\lambda^3 (\lambda p +1)} \ln \left ( \frac{1}{\lambda + 1} \right ).
\]
The result follows after identifying the parameters $k_y, m_y$ and $\nu$ in item $iv)$ of Lemma \ref{lemaux}.

\end{proof}


\end{document}